\documentclass[12pt,a4paper]{article}

\usepackage{amsmath, amsfonts, amsthm, amssymb}
\usepackage{url}
\usepackage{enumerate}

\theoremstyle{plain}
\newtheorem{proposition}{Proposition}
\newtheorem{theorem}[proposition]{Theorem} 
\newtheorem{lemma}[proposition]{Lemma}

\newtheorem{corollary}[proposition]{Corollary}

\theoremstyle{definition}
\newtheorem{definition}[proposition]{Definition}
\newtheorem{remark}[proposition]{Remark}

\newcommand {\Z} {\mathbb{Z}}

\newcommand{\bra}{\langle}
\newcommand{\ket}{\rangle}

\DeclareMathOperator{\Aut}{Aut}
\DeclareMathOperator{\GL}{GL}
\DeclareMathOperator{\ch}{char}
\DeclareMathOperator{\Ind}{Ind}

\DeclareMathOperator{\D}{\mathsf{D}}
\DeclareMathOperator{\F}{\mathbb{F}}

\DeclareMathOperator{\Hom}{Hom}

\newcommand{\hilb}{\F[W]^G_+\F[W]}
\newcommand{\Hilb}{\F[W]^G_+\F[W]_+}

    \title{On the Noether number of $p$-groups}
    \date{}
    \author{K\'alm\'an Cziszter \thanks{Partially supported  by the National Research, Development and Innovation Office (NKFIH)  grants PD113138,  ERC~HU~15 118286, K115799 and K119934.}\\ 
    	\small Alfr\'ed R\'enyi Institute of Mathematics,  Hungarian Academy of Sciences\\
    	\small Re\'altanoda u. 13 -- 15, 1053 Budapest, Hungary }

\begin{document}

\maketitle

\begin{abstract}
A  group of order $p^n$ ($p$ prime) has an indecomposable polynomial invariant of degree at least $p^{n-1}$ if and only if the group  has a cyclic subgroup of index at most $p$ or it is isomorphic to 
the elementary abelian group of order 8 or the Heisenberg group of order 27. 
\end{abstract}

\section{Introduction}

Let $G$ be a finite group and $V$ a $G$-module over a field $\F$ of characteristic  not dividing the group order $|G|$. 
The Noether-number $\beta(G,V)$ is the maximal degree in a minimal generating set 
of the ring of polynomial invariants $\mathbb F[V]^G$. It is known  that $\beta(G,V) \le |G|$  (see \cite{noether}, \cite{fogarty}, \cite{fleishmann}). 
Even more, it was observed that $\beta(G) := \sup_{V}(G,V)$ (where $V$ runs over all $G$-modules over the base field $\F$) is typically much less than  $|G|$. 
For an algebraically closed base field of characteristic zero it was proved in \cite{schmid} that $\beta(G) = |G|$ holds only if $G$ is cyclic. 
Then it turned out that $\beta(G) \le \frac 3 4 |G|$  for any non-cyclic group $G$ (see \cite{DH} and \cite{sezer}). 
Moreover  $\beta(G) \ge \frac 1 2 |G|$ holds  
 if and only if $G$ has a cyclic subgroup of index at most two, with the exception of four particular groups of small order (see \cite[Theorem~1.1]{CzD:3}).
Recently some  asymptotic extensions of this result were given in \cite{HP}. 
Our goal in the present article is to establish  the following  strengthening  of this kind of results for the class of $p$-groups: 

\begin{theorem}\label{fo2}
If $G$ is a finite $p$-group for a prime $p$ and 
the characteristic of the base field $\F$ is zero or greater than $p$  then 
 the inequality
\begin{align} \label{fotul}\beta(G) \ge \frac 1 p |G| \end{align}
holds if and only if $G$ has a cyclic subgroup of index at most $p$ or $G$ is  the elementary abelian group $C_2\times C_2 \times C_2$ or  the Heisenberg group of order 27. 
\end{theorem}

The proof of  Theorem~\ref{fo2} will be reduced to the study of a single critical case, the Heisenberg group $H_p$, which is the extraspecial group of order $p^3$ and exponent  $p$ for an odd prime $p$.
We  prove about this the following result: 

\begin{theorem}\label{beta_H_p}
For any prime $p \ge 5$  and base field $\F$ of characteristic $0$ or greater than $p$ we have $ \beta(H_p) < p^2$.
\end{theorem}

The paper is organised as follows. 
Section~\ref{Sec:zerosum}  contains some technical results on zero-sum sequences over  abelian groups that will be needed later. 
In Section~\ref{Sec:reduction} we reduce the proof of Theorem~\ref{fo2} to that of Theorem~\ref{beta_H_p}. 
Then in Section~\ref{Sec:invariant} we  explain the main invariant theoretic idea behind the proof of Theorem~\ref{beta_H_p}  which is also applicable in a more general setting. 
The proof itself of Theorem~\ref{beta_H_p}   will then be carried out in full detail  in  Section~\ref{Sec:3c}. 
Finally, Section~\ref{Sec:4}  completes our argument by showing that for the case $p=3$   we have  $\beta(H_3)= 9$   in any non-modular characteristic.

\section{Some preliminaries on zero-sum sequences}\label{Sec:zerosum}

We  follow here in our notations and terminology the usage fixed in \cite{CzDG}. 
Let $A$ be an abelian group  noted additively. 
By a sequence $S$ over a subset $A_0 \subseteq A$ we mean  a multiset of elements of $A_0$. 
They form a free commutative monoid with respect to  concatenation, denoted by $S\cdot T$, 
and  unit element  the empty sequence $\emptyset$; this has to be distinguished from $0$, the  zero element of $A$.
The sequence $ a\cdot a \cdots a$, obtained by the $k$-fold repetition of an element $a \in A$, is denoted by $a^{[k]}$; this has to be distinguished
from the product $ka \in A$.
The multiplicity of an element $a\in A$ in a sequence $S$ is denoted by $\mathsf v_a(S)$. 
We also write $a \in S$ to indicate that $\mathsf v_a(S) >0$. 
We say that $T$ is a subsequence of $S$, and write $T\mid S$, if there is a sequence $R$ such that $S = T \cdot R$. 
In this case we also write $R = S \cdot T^{[-1]}$. 
The \emph{length} of a sequence, denoted by $|S|$, can be expressed as $\sum_{a \in A} \mathsf v_a(S)$,
whereas the \emph{sum} of a sequence $S = a_1 \cdots a_n$ is  $\sigma(S) := a_1 + \ldots + a_n \in A$ 
and by convention we set $\sigma(\emptyset) = 0$.
We say that $S$ is a \emph{zero-sum sequence} if $\sigma(S) =0$.

The relevance of zero-sum sequences for our topic is due to the fact  that for an abelian group $A$ the Noether number $\beta(A)$ coincides with the Davenport constant $\D(A)$,
which is defined as the maximal length of a zero-sum sequence over $A$  not containing any non-empty, proper zero-sum subsequence  (see e.g. \cite[Chapter~5]{CzDG}). 
Its value for $p$-groups is given by the following formula  \cite[Theorem~5.5.9]{bible}:
\begin{align}\label{olson_p}
\D(C_{p^{n_1}} \times \dots \times C_{p^{n_r}} )= \sum_{i=1}^r (p^{n_i} -1) +1.
\end{align}

A variant of this notion is the $k$th Davenport constant $\mathsf D_k(A)$ defined for any $k \ge 1$ as the maximal length of a zero-sum sequence $S$
that cannot be factored as the concatenation $S = S_1 \cdots S_{k+1}$ of  non-empty zero-sum sequences $S_i$ over $A$. 
Its numerical value is much less known (for some recent results  see \cite{freeze-schmid}); we shall only need the fact
that  according to \cite[Theorem~6.1.5.2]{bible}:
\begin{align}
\label{halter_koch}
\D_{k} (C_p \times C_p) &= kp +p-1.
\end{align}

The following  consequence of the definition of $\mathsf D_k(A)$ will also be used:

\begin{lemma}[\cite{bible}, Lemma~6.1.2]\label{HK}
 Any sequence $S$ over an abelian group $A$ of length at least $\mathsf D_k(A)$ factors
as $S=S_1 \cdots S_k\cdot R$ with some non-empty zero-sum sequences $S_i$.
\end{lemma}

We define for any sequence $S$ over $A$ the set of all partial sums of $S$  as $\Sigma(S) := \{ \sigma(T): \emptyset \neq T \mid S \} $. 
If $0 \not\in \Sigma(S)$ then $S$ is called \emph{zero-sum free}.
  The next result could also be deduced from the Cauchy-Davenport theorem (see \cite[Corollary~5.2.8.1]{bible}) but  we  provide here an elementary proof for the reader's convenience:

\begin{lemma}\label{CD}
Let $p$ be a prime.
Then for  any  sequence $S$ over $C_p \setminus \{ 0 \}$  we have  $|\Sigma(S)|\ge \min \{p,  |S|\}$.
\end{lemma}
\begin{proof}
We use induction on the length of $S$. 
For $|S| = 0$ the claim is trivial.
Otherwise consider a sequence $S\cdot a$ where the claim holds for $S$. We have   $\Sigma(S \cdot a) = \Sigma(S) \cup \{ a\} \cup (a + \Sigma(S))$, 
where $a+ \Sigma(S) := \{a + s : s \in \Sigma(S) \}$. 
Then either $|\Sigma(S\cdot a)| \ge |\Sigma(S)| +1$, or else  $a \in \Sigma(S)$ and $a+\Sigma(S) = \Sigma(S)$, that is when $\Sigma(S)$ is a subgroup of $C_p$ containing $a$. 
But since $C_p$ has only two subgroups and by assumption $\Sigma(S) \ni a \neq 0$,  this means that $\Sigma(S) = C_p$. 
\end{proof}

\begin{lemma}[\cite{bible}, Theorem 5.1.10.1] \label{zerosumfree}
A sequence $S$ over $C_p$ ($p$ prime) 
of length $|S| =p-1$ is zero-sum free if and only if $S= a^{[p-1]}$ for some $a \in C_p \setminus \{0\}$.
\end{lemma}
\begin{lemma}[\cite{bible}, Proposition~5.7.7.1] \label{eta}
Let $p$ be a prime and $S$ be a sequence over $C_p \times C_p$ of length $|S|\ge 3p-2$.
Then $S$ has a zero-sum subsequence $X \mid S$  of length $p$ or $2p$. 
\end{lemma}

We close this section with a technical result.  Its motivation and  relevance will become apparent  through its application  in the proof of Proposition~\ref{amorf}.
For any  function $\pi$ defined on $A$ and any sequence $S$ over  $A$   we  will write $\pi(S)$  for the sequence obtained from $S$ by applying $\pi$ element-wise.

\begin{lemma} \label{nullak}
Let $S$ be a sequence  over $C_p$  
of length  $|S| \ge p^2-1$. 
If  we have $\mathsf v_0(S)\ge p+1$ then  $S= S_1 \cdots S_{\ell} \cdot R$, 
where  each $S_i$ is a non-empty zero-sum sequence and $\ell \ge 2p-1$. 
\end{lemma}

\begin{proof} 
Let $\ell$ denote the maximal integer such that $S =S _1 \cdots S_{\ell} \cdot R$ 
for some  non-empty zero-sum sequences $S_i$. Then each $S_i$ is irreducible, hence $|S_i| \le  p$ and $R$ is zero-sum free, hence $|R|\le p-1$. 
Assuming that $\ell \le 2p-2 $  we get
\[ p^2-1 \le |S| \le \mathsf v_0 (S) + (\ell - \mathsf v_0(S))p + p-1 \le (p-1) (2p+1 - \mathsf v_0(S)) \] 
whence $\mathsf v_0(S) \le p$ follows, in contradiction with our assumption.
\end{proof}

\begin{proposition} \label{separ} 
Let $A= C_p \times C_p$ for some prime $p \ge 5$ and $\pi:A \to C_p$ the projection onto the first component.
If $S$ is a sequence over $A$ with $|S| \ge p^2-1$ and $\mathsf v_0(\pi(S)) \le p$ then 
for any given  subsequence  $T \mid S $  of length $|T| \le p-1$ there is a
 a factorisation  $S = S_1 \cdots S_{p-1}\cdot R$, where each $S_i$ is a non-empty zero-sum sequence over $A$, 
while $T \mid S \cdot (S_1 \cdot S_2)^{[-1]}$ and $\Sigma(\pi(S_1)) = C_p$.

\end{proposition}

\begin{proof}
Let $S^* \mid S \cdot T^{[-1]}$ be the maximal subsequence such that $0 \not\in \pi(S^*)$.
Then by assumption $|S^*| \ge |S| -2p +1\ge 3p$, as $p \ge 5$, so there is a zero-sum subsequence $X \mid S^*$ of length $p$ or $2p$  by Lemma~\ref{eta}. We have two cases:

(i) If $|X| = 2p$ then  
$X =S_1\cdot S_2$ for some non-empty zero-sum sequences $S_1, S_2$
such that $|S_1| \ge p$ and $|S_2| \le p$, as $\mathsf D(C_p \times C_p) =2p-1$ by \eqref{olson_p}. 
 
(ii) If $|X| = p$ then we can take $S_1 := X$. 
Then we have $|S \cdot (S_1\cdot T)^{[-1]}| \ge |S| -2p+1 \ge 3p$, so again by Lemma~\ref{eta} we  find a non-empty zero-sum sequence $S_2 \mid S \cdot (S_1\cdot T)^{[-1]}$
of length $|S_2| \le p$ as above. 

In both cases $T \mid S \cdot (S_1 \cdot S_2)^{[-1]}$ and $|S_1 \cdot S_2| \le 2p$ by construction.
Consequently $|S \cdot (S_1 \cdot S_2)^{[-1]}| \ge |S| -2p \ge p^2-2p-1 = \mathsf D_{p-3}(C_p\times C_p)$,
hence by Lemma~\ref{HK} we have a factorisation $S \cdot (S_1 \cdot S_2)^{[-1]} = S_3 \cdots S_{p-1} \cdot R$
with non-empty zero-sum sequences $S_i$ for each $i \ge 3$. 
Finally, in both cases we had $|S_1| \ge p$ and $0 \not \in \pi(S_1)$, hence $|\Sigma(\pi(S_1))| = p$ by Lemma~\ref{CD}.
\end{proof}

\section{Reduction of Theorem~\ref{fo2} to Theorem~\ref{beta_H_p}}\label{Sec:reduction}

Our main tool here will be the $k$th Noether number $\beta_k(G,V)$
which is defined for any $k \ge 1$ as the greatest integer $d$ such that
some invariant of degree $d$ exists which is not contained in
the ideal of $\F[V]^G$ generated by the products of at least $k+1$ invariants of positive degree. 
This notion  was introduced in \cite[Section~2]{CzD:1}
with the goal of  estimating the ordinary Noether number from information on its composition factors. This was made possible by \cite[Lemma~1.4]{CzD:3} 
according to which for any normal subgroup $N \triangleleft G$ we have:
\begin{align}
\label{reduction}
\beta(G,V) &\le \beta_{\beta(G/N)}(N,V). 
\end{align}
As observed in \cite[Chapter~5]{CzDG}, if $A$ is an abelian group then $\beta_k(A)$ coincides with   $\D_k(A)$, 
so that we can use \eqref{halter_koch} in the applications of \eqref{reduction}.

\begin{proof}[Proof of Theorem~\ref{fo2} (assuming Theorem~\ref{beta_H_p})]
The ``if'' part follows from  \cite[Proposition~5.1]{schmid} which states that $\beta(C) \le \beta(G)$ for any  subgroup $C \le G$. 
So if $C$ is cyclic of index at most $p$ then $ \beta(G) \ge \beta(C)= |C|=|G|/[G:C] \ge  \frac 1 p |G|$.
Moreover $\beta(C_2^3) =4$ by \eqref{olson_p} and $\beta(H_3) \ge 9$ by Proposition~\ref{H3_felso} below.

The ``only if'' part for $p=2$ follows from \cite[Theorem~1.1]{CzD:3} so for the rest we may assume that $p \ge 3$.
Let $G$ be a  group of order $p^n$ for which  \eqref{fotul} holds.
If $G$ is non-cyclic then it  has a normal subgroup $N \cong C_p \times C_p$ by
\cite[Lemma~1.4]{berk}. We claim that $G/N$ must be cyclic. For otherwise  by applying \cite[Lemma~1.4]{berk}
to the factor group $G/N$ we find a subgroup $K$ such that 
$N \triangleleft K \triangleleft G$  and $K/N \cong C_p \times C_p$.
But then we get using \eqref{reduction} and \eqref{halter_koch} that
\[ \beta(K) \le \beta_{\beta(C_p \times C_p)}(C_p \times C_p) = p(2p-1)+p-1 = 2p^2-1 < p^3=\frac 1 p |K|.\]
As $\beta(G)/|G| \le \beta(K)/|K| $ by \cite[Lemma~1.2]{CzD:3} we get a contradiction with  \eqref{fotul}.

Now let $g \in G$ be such that $gN$ generates $G/N \cong C_{p^{n-2}}$. Then $g^{p^{n-2}} \in N$ has order $p$ or $1$.
In the first case $\bra g \ket$ has index $p$ in $G$ and we are done.
In the other case $\bra g \ket \cap N = \{1\}$ hence $G \cong N \rtimes \bra g \ket$. 
If $g$ acts trivially on $N$ then $G$ contains a subgroup $H \cong C_p \times C_p \times C_p$
for which  we have  $\beta(H) = 3p -2$ by \eqref{olson_p} 
hence $\beta(G)/|G| \le \beta(H)/|H| < 3/p^2 \le 1/p$, as $p\ge 3$, a contradiction. 
This shows that $g$ must act non-trivially on $C_p \times C_p$.
It is well known that $\Aut(C_p \times C_p) = \GL(2,p)$ has order $ (p^2-1)(p^2-p)$,
so its Sylow $p$-subgroup must have order $p$ and it is isomorphic to $C_p$.
Therefore $g^p$ must act trivially on $N$, so if $n \ge 4$ then $g^p \neq 1$ and the subgroup $\langle N, g^p\rangle $ is isomorphic to $ C_p \times C_p \times C_p$, but this  was excluded before. 
The only case which remains open is that $n = 3$ and $G \cong (C_p \times C_p) \rtimes C_p$, 
where the factor group $C_p$ acts non-trivially on $C_p \times C_p$.
This is the Heisenberg group  denoted by  $H_{p}$. 
By Theorem~\ref{beta_H_p} we have $\beta(H_p) < p^2$ for all $p>3$ under our assumption on the characteristic of the base field $\F$. 
So among the Heisenberg groups the inequality \eqref{fotul} can only hold for $H_3$. 
\end{proof}

\begin{remark}
The precise value of the Noether number is already known for all the $p$-groups which satisfy \eqref{fotul}   according to Theorem~\ref{fo2}. 
As the Theorem states,  equality holds in  \eqref{fotul} for $C_2^3$ and $H_3$. 
For the rest, the groups of order $p^n$  which have a cyclic subgroup of index  $p$
were classified  by  Burnside (see e.g. \cite[Theorem~1.2]{berk}) as follows:

(i) if $G$ is abelian, then either $G$ is cyclic with $\beta(G) = p^n$ or  $G = C_{p^{n-1}} \times C_p$ in which case it has $\beta(G) = p^{n-1} +p -1$ by \eqref{olson_p}

(ii) if $G$ is non-abelian and $p > 2$ then $G$ is isomorphic to the modular group $M_{p^n} \cong C_{p^{n-1}} \rtimes C_p$. 
We have  $\beta(M_{p^n}) = p^{n-1} + p-1$ by \cite[Remark~10.4]{CzD:2}. 

(iii) if $G$ is non-abelian and $p=2$ then $G$ is  the dihedral group $D_{2^n}$ or the semi-dihedral group $SD_{2^n}$ or the generalised quaternion group $Q_{2^n}$. 
We have $\beta(Q_{2^n}) = 2^{n-1} +2$ and $\beta(D_{2^n}) = \beta(SD_{2^n}) = 2^{n-1}+1$
by   \cite[Theorem~10.3]{CzD:2}.

Altogether these results imply that  for  any non-cyclic $p$-group $G$  we have \begin{align}\beta(G) \le \frac 1 p |G| + p \end{align}
and this inequality is sharp only for the case  $p= 2$.
\end{remark}

\begin{remark} 
The notion of the Davenport constant $\mathsf D(G)$, originally defined only for abelian groups as in Section~\ref{Sec:zerosum}, was extended to any finite group $G$ in \cite{GeGryn, Gryn}. 
For the conjectural connection between  the Noether number and this generalisation of the Davenport constant  see \cite[Section~5.1]{CzDG} and \cite{CzDSz}.
\end{remark}

\section{Invariant theoretic lemmas}\label{Sec:invariant}

Let us fix here some notations related to invariant rings. For any vector space $V$ over a field $\F$ we denote its coordinate ring by $\F[V]$. 
We say that a group $G$ has a left action on $V$, or that $V$ is a $G$-module, if a group homomorphism $\rho: G \to \GL(V)$ is given
and we abbreviate $\rho(g)(v)$ by writing $g \cdot v$ for any $g \in G$ and $v \in V$. 
By setting $f^g(v) := f(g\cdot v)$ for any $f \in \F[V]$ we obtain a right action of $G$ on $\F[V]$. 
The ring of polynomial invariants is defined as $\F[V]^G := \{f \in \F[V]: f^g =f \; \text{ for all } g \in G \}$.
If the ring $\F[V]^N$ is already known for some normal subgroup $N \triangleleft G$
then $\F[V]^{G}$  as a vector space is spanned by its elements of the form $\tau_N^G(m)$,
where $m$ runs over the set of all monomials and 
$\tau_N^G: \F[V]^N \to \F[V]^{G}$ is the $\F[V]^G$-module epimorphism defined as \[\tau_N^G(m) = \frac 1 {|G/N|} \sum_{g \in G } m^g\] (see e.g. \cite[Chapter 2.2]{NeuselSmith}).
When $N$ is trivial this definition amounts to the Reynolds operator $\tau := \tau_{\{ 1\}}^G$. 
Given any character $\chi \in \widehat{G}:=\Hom(G, \F^{\times})$ the set   $\F[V]^{G,\chi} := \{ f\in \F[V] : f^g = \chi(g) f\}$
constitutes the $\F[V]^G$-module of $G$-semi-invariants of weight $\chi$.
If the restriction of $\chi$ to $N$ is trivial, i.e. when $\chi \in \widehat{G/N}$, then 
these semi-invariants can be obtained by the  projection map $\tau_{\chi}: \F[V]^N \to \F[V]^{G,\chi}$  defined with the analogous formula \[\tau_{\chi}(u) = \frac 1 {|G/N|}\sum_{g \in G/N} \chi^{-1}(g) u^g.\] 

$\F[V]$ and $\F[V]^G$ are  graded rings:  $\F[V]_d$ denotes for any $d \ge 0$  the vector space of degree $d$ homogeneous polynomials  and $\F[V]^G_d = \F[V]^G \cap \F[V]_d$.
The set $\F[V]^G_+ := \bigoplus_{d \ge 1} \F[V]^G_d$ is a maximal ideal in $\F[V]^G$, while $\F[V]^G_+\F[V]$, the ideal of $\F[V]$ generated by all $G$-invariant polynomials of positive degree, 
is the so called \emph{Hilbert-ideal}. This ideal will be  our main object of interest 
since, as  observed in \cite[Section~3]{CzD:1}, the graded factor ring $\F[V]/\F[V]^G_+\F[V]$ is finite dimensional and its top degree, denoted by $b(G,V)$, yields an upper bound on the Noether number
by an easy argument using the Reynolds operator:
\begin{align}\label{b_beta}
\beta(G,V) \le b(G,V) +1.
\end{align}

It is well known that $\beta(G,V)$ is unchanged when we extend the base field 
so we will assume throughout this paper that $\mathbb{F}$ is algebraically closed.

\begin{lemma}\label{b+1} 
Let $G$ be a finite group with a normal subgroup $N$ such that $G/N$ 
is abelian.  Let $W$ be a $G$-module over $\F$ and assume that $|G| \in  \F^{\times}$. 
Then $(\F[W]_+^N)^k  \subseteq \hilb$ for any $k \ge \mathsf D(G/N)$. 
\end{lemma}

\begin{proof}

 $\F[W]^N$ regarded as a $G/N$-module  has the direct sum decomposition $\bigoplus_{\chi \in \widehat{G/N}} \F[W]^{G,\chi}$. 
(Here  we used both our assumptions on $\F$.) 
This means that any element $u \in \F[W]^N_+$ can be written as a sum $u = \sum_{\chi \in \widehat{G/N}} \tau_{\chi}(u)$.
Now  for any  $k \ge 1$ and $u_1,\ldots,u_{k} \in \F[W]_+^N$ we have
\begin{align}\label{trukk:1}
\prod_{i=1}^{k}u_i = 
\prod _{i=1}^{k} \left(\sum_{\chi \in \widehat{G/N}} \tau_{\chi}(u_i) \right) = 
\sum_{\chi_1,\ldots,\chi_{k} \in \widehat{G/N}} \tau_{\chi_1}(u_1) \cdots \tau_{\chi_{k}}(u_{k}).
\end{align}
The term $\tau_{\chi_1}(u_1) \cdots \tau_{\chi_{k}}(u_{k})$ belongs to the ideal $\hilb$
whenever the  sequence $(\chi_1,\ldots,\chi_{k})$  over $\widehat{G/N} \cong G/N$ contains a non-empty zero-sum subsequence.  
But this holds for every term on the right  of \eqref{trukk:1} as $k \ge \mathsf D (G/N)$. 
\end{proof}

\begin{lemma}\label{trukk}
If in  Lemma~\ref{b+1} the factor group $G/N \cong C_p$ is cyclic of prime order then for any $g\in G/N$  
and any elements $u_1,\ldots,u_{p-1} \in \F[W]^N_+$ we have the relation:
\begin{align} \label{trukk:2} u_1 \cdots u_{p-1} - u_1^{g}u_2^{-g}u_3 \cdots u_{p-1} \in \hilb.\end{align}

\end{lemma}

\begin{proof}
Observe that in \eqref{trukk:1} with $k = p-1$ the weight sequence $(\chi_1,\ldots,\chi_{p-1})$ over $\widehat{C}_p$ is zero-sum free
if and only if  $\chi_1=\ldots=\chi_{p-1}$ and $\chi_1$ is non-trivial (by Lemma~\ref{zerosumfree}). 
As a result we get:
\begin{align*}
u_1\cdots u_{p-1} \in  \sum_{\chi \in \widehat{C}_p \setminus \{ 1 \}} \tau_{\chi}(u_1) \cdots \tau_{\chi}(u_{p-1})
+ \hilb.
\end{align*}
Replacing here $u_1$ and $u_2$ with $u_1^g$ and $u_2^{-g}$, respectively,
and observing that by  definition we have $\tau_{\chi}(u^g) = \chi(g)\tau(u)$ for any $u\in \F[W]^N$ we infer that
$u_1^g u_2^{-g} u_3\cdots u_{p-1}$ must belong to the same residue class modulo the ideal $\hilb$ to which $u_1\cdots u_{p-1}$ does belong. This proves our claim.
\end{proof}

\section{The Heisenberg group $H_p$}\label{Sec:3c}

The Heisenberg group $H_p = \bra a, b \ket$ can be defined  by the presentation:
\begin{align}\label{H_rel}
a^p=b^p=c^p =1 \quad
[a,b] =c \quad
[a,c] =[b,c] =1
\end{align}
where $[a,b]$ denotes the commutator $a^{-1} b^{-1} ab$.
The  subgroups $A:= \bra a,c \ket$ and $B:= \bra b,c \ket$ are normal and isomorphic to $C_p \times C_p$. 
The Frattini-subgroup, the center and the derived subgroup of $H_p$ all coincide with $\bra c\ket$,
so that $H_p$ is extraspecial. 
In particular  $H_p / \bra c \ket$ is also isomorphic to $C_p \times C_p$.
Taking into account only the subgroup  structure of $H_p$ the best upper bound that we can give about its Noether number
 by means of \eqref{reduction} and \eqref{halter_koch} is the following:
\begin{align}\label{apriori} 
\beta(H_p) \le \beta_{\beta(C_p)}(C_p \times C_p) = p^2 + p -1. \end{align}
Our goal in this section will be to enhance this estimate by analysing more closely the invariant rings of  $H_p$.

Let $\F$ be an algebraically closed field with $\ch(\F) \neq p$, so that there is a primitive $p$-th root of unity $\omega \in \F$ that  will be regarded as fixed  throughout this paper.
The irreducible $H_p$-modules over $\F$ are then of two types:
 
(i) Composing any group homomorphism $\rho \in \Hom (C_p \times C_p, \F^{\times} )$ 
with the canonic surjection $H_p \to H_p/\bra c\ket \cong C_p \times C_p$  
yields $p^2$ non-isomorphic $1$-dimensional irreducible representations of $H_p$. 

(ii) 
For each  primitive $p$-th root of unity $\omega^i \in \F$, where $i=1,\ldots,p-1$, take
the induced representation $V_{\omega^i} := \Ind_A^{H_p}\bra v \ket$, 
where $\bra v \ket$ is a $1$-dimensional left $A$-module
such that $a\cdot v =v$ and  $c \cdot v = \omega^i v$.
In the  basis $\{v, b \cdot v, ..., {b^{p-1}} \cdot v \}$ this representation is then given in terms of matrices in the following form,
with $I_p$ the $p \times p$ identity matrix:
\begin{align} \label{H_irred}
 a \mapsto \left(\begin{array}{cccc}1 &   &   &   \\  & \omega^i &   &   \\  &   & \ddots &   \\  &   &   &  \omega^{i(p-1)}\end{array}\right)
\quad
b \mapsto \left(\begin{array}{cccc} 0&  \cdots & \cdots & 1 \\1 &  &  & \vdots \\ & \ddots &  &  \vdots\\ &  & 1 & 0\end{array}\right)
\quad
c \mapsto \omega^i I_p.
\end{align}
Each $V_{\omega^i}$ is irreducible by Mackey's criterion (see e.g. \cite{serre}) and for $\omega^i \neq \omega^{i'}$ it is easily seen
(e.g. from the matrix corresponding to $c$) that $V_{\omega^i}$ and $V_{\omega^{i'}}$ are non-isomorphic as $G$-modules.

Adding the squares of the dimensions of the above irreducible $H_p$-modules we get $p^2\cdot 1 + (p-1)p^2 =p^3= |H_p|$,
so that no other irreducible $H_p$-modules exist. 
As a result 
 an arbitrary $H_p$-module $W$ over  $\F$ has the canonic 
 direct sum decomposition 
\begin{align}\label{direct}
W =  U  \oplus V_1  \oplus \ldots \oplus V_{p-1} 
\end{align}
 where $U$ consists only of $1$-dimensional irreducible representations of $H_p$ with $\bra c \ket$ in their kernel,
while each $V_i$ is an isotypic $H_p$-module consisting of  the direct sum of $n_i \ge 0$ isomorphic copies of the irreducible representation $V_{\omega^i}$:
\begin{align}\label{isotypic}
V_i = \underbrace{V_{\omega^i} \oplus \ldots \oplus V_{\omega^i}}_ {n_i \text{ times}}.
\end{align}

Next we recall how does the action of $G$ on $W$ extend to the coordinate ring $\F[W]$. 
When speaking of a coordinate ring $\F[V_{\omega^i} ] = \F[x_{i,0},...,x_{i,p-1}]$ 
we  always tacitly assume that 
the variables $x_{i,k}$ form a dual basis of the  basis used at \eqref{H_irred}. 
By our convention from Section~\ref{Sec:invariant},  $H_p$ acts  from the right on the variables, i.e. $x^g(v) = x(g \cdot v)$ for all $g\in H_p$, 
so we can rewrite \eqref{H_irred} as:
\begin{align}\label{action}
x_{i,k}^b &= x_{i,(k-1) \mathrm{mod } \, p} \qquad \qquad 
x_{i,k}^a = \omega^{ik} x_{i,k} \qquad \qquad 
x_{i,k}^c = \omega^{i} x_{i,k} .
\end{align}
(Here, by some abuse of notation, we identified the integers $k=0,1,\ldots,p-1$ occurring as indexes with the modulo $p$ residue classes they represent.) 
This shows that the action of the subgroup $A$ on a variable $x_{i,k}$ is completely determined by the modulo $p$ residue classes of the exponents $ik$ and $i$ of $\omega$ in \eqref{action};
we will call  $\phi(x_{i,k}) := (ik , i) \in \Z / p \Z \times \Z / p \Z$ the \emph{weight} of the variable $x_{i,k}$.
We shall also refer to the projections $\phi_a(x_{i,k})= ik $ and $\phi_c(x_{i,k}) =i$.
With this notation it is  immediate from \eqref{action} that  for any $n \in \Z$ and $x=x_{i,k}$ 
\begin{align}\label{action2}
\phi_a(x^{b^n}) = \phi_a(x) - n \, \phi_c(x)
\quad \text{ and } \quad 
\phi_c(x^{b^n}) = \phi_c(x)
\end{align}
where the subtraction and multiplication with $n$ is understood in  $\Z / p \Z$.
This implies the observation, which will be used frequently  later on,  
that for any variable $x$ with $\phi_c(x) \neq 0$ and any arbitrarily given  $w \in \Z / p \Z$ 
there is always an element $g \in \bra b \ket$ such that $\phi_a(x^g) = w$.  
Our discussion also shows that  for a variable $y \in \F[W]$ we have $\phi_c(y) = 0$ if and only if $y \in \F[U]$,
and otherwise the value $\phi_c(y) =i $ determines the isotypic $H_p$-module $V_i$ such that $y \in \F[V_i]$.

Any monomial  $u\in \F[W]$ is  an $A$-eigenvector, too, hence we can associate a weight $\phi(u) := (j,i) \in \Z / p \Z \times \Z / p \Z$ to it
so that  $u^a = \omega^ju$ and $u^c = \omega^i u$. Obviously then $\phi(uv) = \phi(u) + \phi(v)$ for any  monomials $u,v$.
If  $u =  y_1 \cdots y_n$ for some variables $y_i \in \F[W]$, with repetitions allowed, then we can form 
the sequence  $\Phi(u) := \phi(y_1) \cdots \phi(y_n)$ over $A$,
which  will be  called the \emph{weight sequence} of $u$. 
 Obviously  $\phi(u) = \sigma(\Phi(u)) = \phi(y_1) + \cdots+\phi(y_n)$ with the notations of Section~\ref{Sec:zerosum}. 
Observe that a monomial $u$ is $A$-invariant if and only if $\phi(u) =0$, that is if $\Phi(u)$ is a zero-sum sequence over $A$.
Finally, we set  $\Phi_a(u) := (\phi_a(y_1), \ldots, \phi_a(y_n)) $ and $\Phi_c(u) := (\phi_c(y_1), \ldots, \phi_c(y_n)) $.

\begin{definition}
We call two monomials $u,v\in \F[W]$ \emph{homologous},  denoted by $u \sim v$,
if $\deg(u) = \deg(v) = d$ and $u= \prod_{n=1}^d y_n$ while $v = \prod_{n=1}^d y_n^{g_n}$ for  some  variables $y_n \in \F[W]$ (with  repetitions  allowed) and group elements $g_n \in \langle b \rangle$.
\end{definition}

Observe that a monomial $v$   obtained from a monomial $u$ by repeated applications of \eqref{trukk:2}  will be homologous  to it in the above sense. 

\begin{proposition}\label{amorf} 
Let $p \ge 5$. If $u \in \F[W]$ is a monomial with $\deg(u) \ge p^2- 1$, $\mathsf v_0(\Phi_c(u)) \le p$
and  $v \mid u$ is a monomial such that $\deg(v) \le p $ and $0 \not\in \Phi_c(v)$
then for  any homologous monomial $v' \sim v$
there is a homologous monomial $u' \sim u$ such that  $v' \mid u'$ and   $u'-u \in \hilb$.
\end{proposition}

\begin{proof}
We use induction on the degree $d:= \deg(v) = \deg(v')$.
If $d =0$ then  $v=v'=1$, so we are done by taking $u'=u$.
Suppose now that the claim holds for some $d \le  p-1$. 
It suffices to prove that for any given divisor $xv \mid u$, where $x$ is a variable, $\deg(v) = d$,   $0 \not\in \Phi_c(xv)$, and for  any    $v'\sim v$ and $g \in \langle b \rangle$ 
a monomial $u'' \sim u$ exists such that $x^gv' \mid u''$ and $u'' - u \in \hilb$.

By the inductive hypothesis we already have a monomial $u' \sim u$ such that $v' \mid u'$ and   $u'-u \in \hilb$.
As $u'/v' \sim u/v$ and $x$ divides $ u/v$ there is a $t \in \langle b \rangle$ such that $x^t$ divides $u'/v'$.
By applying  Proposition~\ref{separ} to the weight sequences $S := \Phi(u')$, $T := \Phi( v')$
we obtain  a factorisation $u' =u_1 \cdots u_{p-1}u_p$ such that $u_i \in \F[W]_+^A$ for all $i=1,\ldots, p-1$, $u_p \in \F[W]$,   $v'$ divides $u'/u_1u_2$ and $\Sigma(\Phi_c(u_1)) = \Z / p \Z$.
We have two cases:

i) If $x^t \mid u_1$  (or similarly if $x^t \mid u_2$) then take $u'' := u_1^{-t+g}u_2^{t-g}u_3 \cdots u_{p-1}u_p$. 
We have $x^gv' \mid u''$ and $u'' \sim u' \sim u$, while  $u''-u' \in \hilb$ by Lemma~\ref{trukk}. 

ii) Otherwise $x^t \mid u_k$ for some $k>2$. By our assumption on $\Sigma(\Phi_c(u_1))$  there is a divisor $w \mid u_1$ with $\phi_c(w) =  -\phi_c(x^t) $.
As
$\phi_c(x^t) = \phi_c(x) \neq 0$
 there is an $h \in \langle b \rangle$ for which $\phi_a(w^h) = -\phi_a(x^t)$. 
Then for  $\hat u := u_1^hu_2^{-h} u_3 \cdots u_{p-1}r$ we have $\hat u \sim u$ and $\hat u -u \in \hilb$ by Lemma~\ref{trukk}.
Take the factorisation $\hat u = \hat u_1 \cdots  \hat u_p$ where
$\hat u_1 = w^hx^t$,
$\hat u_2 = u_2^{-h}$,
$\hat u_k = (u_k/x^t) (u_1^h/w^h)$
and $\hat u_i = u_i$ for the rest. 
By construction $\hat u _i \in \F[W]^A_+$ for all $i \le p-1$, $v'$ divides $\hat u / \hat u_1 \hat u_2$  and $x^t \mid \hat u_1$, 
so this factorisation of $\hat u$   falls under case i) and we are done. 
\end{proof}

We  need some further notations. 
The  decomposition \eqref{direct} induces an 
 isomorphism $\F[W] \cong  \F[U] \otimes \F[V_1] \otimes \ldots \otimes \F[V_{p-1}] $ 
 which in turn yields for any monomial $m \in \F[W]$ a  factorisation 
 $m=m_0 m_1 \cdots m_{p-1}$ such that $m_0 \in \F[U]$ and $m_i \in \F[V_i]$ for all $i$.
Then  for each $i$ the  decomposition \eqref{isotypic} gives  
the identifications $\F[V_i] = \bigotimes_{j=1}^{n_i} \F[V_{\omega^i}]=\F[x_{i,k}^{(j)}: k=0,\ldots,p-1; j=1,\ldots,n_i]$,
where we set $x_{i,k}^{(j)} := 1 \otimes \cdots \otimes x_{i,k}\otimes \cdots \otimes 1$, i.e. the variable $x_{i,k}$ introduced at \eqref{action} is placed in the $j$th tensor factor.
So for any monomial $m_i \in \F[V_i]$ we have a  factorisation
 $m_i = m_i^{(1)} \cdots m_i^{(n_i)}$ where each monomial $m_i^{(j)}$ depends only on the set of variables $\{x_{i,k}^{(j)}$: $k = 0,1,\ldots, p-1\}$.
 Observe finally that  two monomials $u,v \in \F[V_1 \oplus \ldots \oplus V_{p-1}]$ are homologous, $u \sim v$, if and only if  $\deg(u_{i}^{(j)}) = \deg(v_{i}^{(j)})$ for all $i=1,\ldots,p-1$ and $j=1,\ldots, n_i$. 
 
We shall also need the polarisation operators  defined for any polynomial $f \in \F[W]$ by the formula
\begin{align}
\Delta_i^{s,t}(f) := \sum_{k=0}^{p-1}  x^{(t)} _{i,k}\partial^{(s)}_{i,k}f
\end{align} 
where $\partial^{(s)}_{i,k}$ denotes partial derivation with respect to  the variable $x_{i,k}^{(s)}$.
All  polarisation operations $\Delta:=\Delta^{s,t}_i$ are degree preserving, $\deg(\Delta(f) ) = \deg(f)$, and $G$-equivariant, i.e. $\Delta(f^g) = \Delta(f)^g$. 
Therefore by the Leibniz rule
\begin{equation}\label{Leibniz}
\begin{aligned}
\Delta(\hilb)  & \subseteq \hilb \quad \text{ and }  \\ \Delta(\Hilb) & \subseteq \Hilb.
\end{aligned}
\end{equation}

\begin{proposition}\label{majdnem_tiszta}
Let $p \ge 5$ and assume that $\mathrm{char}(\F)$ is $0$ or greater than $p$. 
 If a monomial $m \in \F[W]$   has $\deg(m) \ge p^2-1$  then  $m \in \hilb$. 
\end{proposition}

\begin{proof}
Consider the factorisation $m= m_0 m_1 \cdots m_{p-1} $ derived from \eqref{direct} as described above.
Observe that for the weight sequence $S = \Phi(m)$ we have $\mathsf v_0(\Phi_c(m)) = \deg(m_0)$. 
So if  $\deg(m_0) \ge p+1$ then $m \in (\F[W]^{\langle c \rangle}_+)^{2p-1}\F[W]$ by Lemma~\ref{nullak} 
and we are done, as $\mathsf D(G/\langle c \rangle) = \mathsf D (C_p \times C_p) = 2p-1$ by \eqref{olson_p} hence $(\F[W]^{\langle c \rangle}_+)^{2p-1} \subseteq \hilb$ by Lemma~\ref{b+1}.

It remains that $\deg(m_0) \le p$.
Then  we must have $\deg(m_i) \ge p$ for some  $i\ge 1$, say $i=1$, 
as otherwise $\deg(m) \le \deg(m_0) +(p-1)^2 \le p^2-p+1$ would follow.
Take the factorisation $m_1 = m_{1}^{(1)} \cdots m_{1}^{(n_1)}$  corresponding to the direct decomposition \eqref{isotypic}. 
We  proceed by induction on   $\mu(m) := \max_{j=1}^{n_1} \deg(m_{1}^{(j)})$.

Assume first that $\mu(m) \ge p$. 
This means that $\deg(m_{1}^{(j)}) \ge p$ for some  $j$, say $j=1$. 
 Now let $v$  be an arbitrary divisor of $ m_{1}^{(1)}$ with degree $\deg(v) = p$ and 
 let $v' = \prod_{g \in \langle b \rangle} x^g$ for some variable $x \in \F[V_{\omega^1}^{(1)}]$.
 Then $v'$ is $b$-invariant by construction.
Moreover by \eqref{action2} we have $\phi_c(v') = p\phi_c(x) =0$ and $\phi_a(v') = p\phi_a(x) - (1+2+\cdots+p-1)\phi_c(x) =0$,
and consequently $v'$ is $G$-invariant.
Now as  $v' \sim v$, we can find by Proposition~\ref{amorf} a monomial $m' \sim m$ such that $m-m' \in \hilb$ and $v' \mid m'$. 
 But then $m' \in \hilb$ and we are done for this case.

Now let $\mu(m) < p$.  As $\deg(m_1)\ge p$,  
 we can take a divisor $v \mid m_1$ such that $v = v^{(i)} v^{(j)}$ for some indices $i \neq j \leq n_1$ where we have  $\deg(v^{(i)}) = \mu(m)$ and $\deg(v^{(j)} )  =1$. 
Then the monomial $v' := (x^{(i)}_{1,1} )^{\mu(m)}x_{1,1}^{(j)}$ is homologous with this $v$
and consequently, by Proposition~\ref{amorf},  a monomial $m' \in \F[W]$ exists such that $m-m' \in \hilb$ and $v' \mid m'$.
Our claim  will now follow  by proving that $m' \in \hilb$.

To this end observe that for the monomial $\tilde m := x_{1,1}^{(i)}m'/x_{1,1}^{(j)}$  we have  $\mu(\tilde m) = \mu(m) +1$,  hence  by the induction hypothesis $\tilde m \in \hilb$ already holds.
Moreover $ \Delta_1^{i,j}(\tilde m) = (\mu(m)+1) m'$  by construction, 
hence $(\mu(m) +1)m ' \in \hilb$ by \eqref{Leibniz} and we are finished because by our assumption on  $\F$
 we are allowed to divide by $\mu(m)+1 \le p < \mathrm{char} (\F )$. 
\end{proof}

\begin{proof}[Proof of Theorem~\ref{beta_H_p}]
From Proposition~\ref{majdnem_tiszta}  we see that 
$\F[W]$ as a module over $\F[W]^G$ is generated by elements of degree at most $p^2 -2$. 
Equivalently, for the top degree in the factor ring $\F[W]/\F[W]_+^G\F[W]$ we have the estimate $b(G,W) \le p^2-2$,
whence by \eqref{b_beta} we conclude that $\beta(G,W)\le p^2-1$. 
\end{proof}

\section{The case p=3}\label{Sec:4}

\begin{proposition}\label{H3_also}
Consider  $V=V_{\omega}$   for a primitive third root of unity  $\omega \in \F$ 
as  given by \eqref{H_irred}. 
Then $\beta(H_3, V)\ge 9$.
\end{proposition}
\begin{proof}
Let $\F[V] = \F[x,y,z]$ with the variables conforming our conventions. 
$\F[V]^{H_3}$ is spanned by the elements  $\tau(m):= \tau_A^{H_3}(m)=\frac 1 3(m+m^b + m^{b^2})$ where $m$ is any $A$-invariant monomial. An easy argument shows that $xyz, x^3, y^3, z^3$ are the only irreducible $A$-invariant monomials. 
Then by enumerating all $A$-invariant monomials of degree  at most $8$ we see  that 
they have  degree  $3$ or $6$ so that for $d \le 8$ we have
$\F[V]^{H_3}_d = R_d$, 
where $ R := \F[xyz, \tau(x^3), \tau(x^3y^3)]$. 
Now if we assume  that $\beta(H_3,V) \le 8$ then $\F[V]^{H_3} = R$ follows. 
Observe however that all the  generators of $R$ are symmetric polynomials, so that $R \subseteq \F[V]^{S_3}$. 
 On the other hand   $\tau(x^6y^3) \in \F[V]^{H_3}$ is not a symmetric polynomial, whence  $\tau(x^6y^3) \not\in R.$ 
 This is a contradiction which proves that  $\beta(H_3,V ) \ge 9$. 
\end{proof}

The upper bound on $\beta(H_3)$ will be obtained by an argument very similar to Propositions \ref{separ}, \ref{amorf} and  \ref{majdnem_tiszta}, 
but since there are  many different details, too, we preferred to give a self-contained treatment of this case here:

\begin{proposition} \label{H3_felso}
If $\mathrm{char}(\F) \neq 3$ then $\beta(H_3) \le 9$.
\end{proposition}

\begin{proof}
Suppose  that $\beta(H_3, W) \ge 10$ holds for a $H_3$-module $W$.  
Then there is a monomial $m \in \F[W]^A$ with $\deg(m) \ge 10$ such that $m \not\in \Hilb$
(as otherwise for any $d \ge 10$ the space $\F[W]^G_d$ spanned by the elements $\tau(m)$ would be contained in  $(\F[W]_+^G)^2$).
Let $S = \Phi_c(m)$, identify $\bra c \ket$ with $\mathbb Z / 3 \mathbb Z$  and let $d_i = \mathsf v_i(S)$ for $i\in \mathbb Z/ 3 \mathbb Z =\{0,1,2\}$. 
Recall that we have the factorisation $m=m_0m_1m_2$ corresponding to the direct decomposition $W = U \oplus V_1 \oplus V_2$,
so that $\deg(m_i) = d_i $ for $i=0,1,2$. 
We may assume by symmetry that $d_1 \ge d_2$.

{\bf A.} {\it We claim that  $d_1 \ge 5$.}

 $S$ is a zero-sum sequence over $\mathbb Z / 3 \mathbb Z$ 
and this is only possible if $d_1 - d_2 \equiv 0 \mod{3}$.
So let $d_1 - d_2 = 3k$ for some integer $k \ge 0$. 
Denoting by $\ell(S)$ the maximum number of non-empty zero-sum sequences into which $S$ can be factored,
we have $\ell(S) = d_0 +d_2 +k \le 5$, as otherwise by Lemma~\ref{b+1} applied with $N = \langle c \rangle$ we get
$m \in (\F[W]_+^{\langle c\rangle})^6 \subseteq \Hilb$, since $H_3/N \cong C_3 \times C_3$ and $\D(C_3^2) = 5$ by \eqref{olson_p}. 
On the other hand 
$ |S| = d_0 +d_1+d_2 \ge 10$. 
Subtracting from this inequality  the previous one yields $ d_1 -k \ge 5$, 
whence the claim.

{\bf B.} {\it For any  $w \mid m_1$ with $\deg(w) \le 2$ 
 there is a factorisation $m=u_1u_2u_3$ with $u_i \in \F[W]^A_+$ such that $w \mid u_3$ and $y \mid u_1$ for some variable $y \mid m_1$. }

As $\deg(m/w) \ge  8 = \D_2(C_3^2)$ there is  a factorisation $m/w= u_1u_2r$ with $u_1,u_2 \in \F[W]^A_+$. 
Setting $u_3 = rw$ enforces  $u_3 \in \F[W]^A_+$. 
Here $\deg(u_3) \le \mathsf D(A)=5$ as otherwise $u_3 \in (\F[W]^A_+)^2$ and $m \in (\F[W]^A_+)^4 \subseteq \Hilb$ by Lemma~\ref{b+1}, a contradiction. 
Therefore we cannot have $m_1 \mid u_3$,  for then by {\bf A.} we have $5 \le \deg(m_1) \le \deg(u_3) \le 5$, so that $m_1 = u_3$ 
and  $\Phi_c(m_1) = 1^{[5]}$, contradicting the assumption that $\Phi(u_3)$ is a zero-sum sequence over $A$.
As a result there is a variable $y \mid m_1$ not dividing $u_3$, whence the claim.

{\bf C.}
{\it For any divisor   $v \mid m_1$  with $\deg(v) \le 3$ and  any  monomial $v' \sim v$
there is a monomial $m' \in \F[W]$ such that $v' \mid m'$ and $m -m' \in \Hilb$.}

Let $v= xw$ and $v'=x^gw'$ where $\deg(w) \le 2$, $w' \sim w$ and $g\in  \langle b \rangle$.
By induction on  $ \deg(v) $
assume that we already have a monomial $m'' \sim m$ such that $m''-m \in \Hilb$ and $x^tw' \mid m''$ for some $t \in \langle b \rangle$. 
According to {\bf B.} there are factorisations  $m'' = u_1u_2u_3$ with  $u_i \in \F[W]^A_+$ such that $w'\mid u_3$ and $y \mid u_1$ for some variable $y \in \F[V_1]$. 
We have two cases: i) If we can take $y = x^t$ in one these factorisations then
 for $m' := u_1^{-t+g}u_2^{t-g} u_3 \sim m''$ we have $m'-m'' \in \Hilb$ by Lemma~\ref{trukk} so  we are done as $v' \mid m'$.
ii) Otherwise necessarily  $x^tw' \mid u_3$  and $y \neq x^t$. Still however, there is an  $h \in \langle b \rangle$ such that  $\phi(y^h) = \phi(x^t)$
hence  for $\tilde m := u_1^hu_2^{-h}u_3 \sim m''$ we have $m-\tilde m \in \Hilb$ by Lemma~\ref{trukk} 
and we obtain a factorisation $\tilde m = \tilde u_1 \tilde u_2 \tilde u_3$ falling under case i) by setting $\tilde u_1 = x^tu_1^h/y^h$, $\tilde u_2 = u_2^{-h}$, $\tilde u_3 = y^hu_3/x^t$, so we are done again.

{\bf D.} Now we proceed as in the proof of Proposition~\ref{majdnem_tiszta}. 
For the sake of simplicity from now on we  rename  our variables so that $\F[V_1] = \bigotimes_{i=1}^{n_1} \F[V_{\omega}]= \F[x_i,y_i,z_i: i=1,\ldots, n_1]$. 
Moreover we abbreviate $\Delta_1^{s,t}$ as $\Delta^{s,t}$.

1) If we have $\deg(m_{1}^{(i)}) \ge 3$ for some $1 \le i \le n_1$ then
we can apply {\bf C.} with $v' := x_iy_iz_i \in \F[W]_+^G$, concluding that $m \in \Hilb$,
a contradiction.

2) Otherwise if $\deg(m_{1}^{(i)}) =2$ for some $i$ then still there is a $j \neq i$ such that $ \deg(m_{1}^{(j)}) \ge 1$. 
After an application of {\bf C.} we may assume that $m$ is divisible by $ x_i^2x_j$.
But then $ m = \frac 1 3 \Delta^{j,i} (\tilde m)$ for the monomial $\tilde m := m x_i/x_j$  which falls under case 1) 
hence $m \in \Hilb$ by \eqref{Leibniz}, a contradiction.

3) Finally, if $\deg(m_1^{(1)}) = \ldots = \deg(m_1^{(n_1)}) =1$ then after an application of {\bf C.} we may assume that 
$x_1y_2z_3 \mid m$. Now consider the relation:
\begin{align}\label{polar}
\Delta^{1,2} (x_1y_1z_3) + \Delta^{2,3} (x_1 y_2 z_2) + \Delta^{3,1} (x_3 y_2 z_3) = 3 x_1y_2z_3 + \tau(x_3y_2z_1)
\end{align}
After multiplying \eqref{polar} with $m' := m / x_1y_2z_3 $ we get on the left hand side
$\Delta^{1,2} ( my_1/y_2) + \Delta^{2,3} (m z_2/z_3 ) + \Delta^{3,1} ( mx_3/x_1) \in \Hilb$ by \eqref{Leibniz}, as all the three monomials occurring here fall under case 2),
and on the right hand side $\tau(x_3y_2z_1)m' \in \Hilb$, whence $3m = 3x_1y_2z_3m' \in \Hilb$ follows.
This contradiction completes our proof.
\end{proof}

Now comparing Proposition~\ref{H3_also} and \ref{H3_felso}   immediately gives:

\begin{corollary}\label{H3}
If $\mathrm{char}(\F) \neq 3$ then $\beta(H_3) = 9$.
\end{corollary}

\begin{remark}
It would be interesting to know if Theorem~\ref{beta_H_p} also extends to  the whole non-modular case,
i.e. for any field $\F$ whose characteristic does not divide $|G|$, just as it is the case for $p=3$ by the above result. 
\end{remark}

\section*{Acknowledgements}
\noindent
The author is grateful to M\'aty\'as Domokos for many valuable comments on the manuscript of this paper.
He also thanks the anonymous referee for many suggestions to improve the presentation of this material.

\end{document}